\documentclass[11pt]{article}
\usepackage[margin=1.1in]{geometry}
\usepackage{amsmath,amssymb,amsthm}
\usepackage{booktabs}
\usepackage[colorlinks=true,linkcolor=blue,citecolor=blue,urlcolor=blue]{hyperref}

\newtheorem{theorem}{Theorem}[section]
\newtheorem{lemma}[theorem]{Lemma}
\newtheorem{corollary}[theorem]{Corollary}
\newtheorem{proposition}[theorem]{Proposition}
\theoremstyle{remark}

\DeclareMathOperator{\mc}{mc}
\newcommand{\bip}{\beta}
\newcommand{\dmono}{d_{\mathrm{mono}}}

\title{The Erd\H{o}s $n^2/25$ max-cut conjecture for small multiples of five,\\
via a per-root-MaxCut envelope and blow-up integrality}
\author{Alper Ferudun\thanks{\texttt{alper@mercurycodelab.com}. The envelope
certificate, its exact-arithmetic verifier, and the brute-force ground-truth
checker accompany this note as ancillary files; development history at
\texttt{github.com/AlperTheKing}.}}
\date{June 2026}

\begin{document}
\maketitle

\begin{abstract}
Erd\H{o}s conjectured that every triangle-free graph on $N$ vertices can be made
bipartite by deleting at most $N^2/25$ edges; the bound would be sharp, attained
by the balanced blow-up $C_5[N/5]$. Writing $\bip(G)$ for the minimum number of
edges whose deletion makes $G$ bipartite and $a(N)=\max\{\bip(G):G$ triangle-free
on $N$ vertices$\}$, the conjecture is $a(N)\le N^2/25$, and for $N=5n$ it reads
$a(5n)\le n^2$. Balogh, Clemen and Lid\'ick\'y proved it for large $N$ in the two
density tails (edge density at most $0.2486$ or at least $0.3197$) and proved the
global bound $a(N)\le N^2/23.5$; the medium-density band remains open. We prove
\[
a(5n)=n^2\qquad\text{for every }1\le n\le 40,\quad\text{i.e. } N\in\{5,10,\dots,200\}.
\]
The proof is computer-assisted and combines three ingredients. (i) A
\emph{per-root-MaxCut envelope}: for the $107$ triangle-free $7$-root types, the
mean over types of the best per-type cut is an upper bound
$\dmono(W)\le U_7(W)$ that is \emph{tight} at the $C_5$-blow-up. (ii) An order-$10$
flag-algebra certificate---the per-root-MaxCut rows at $7$ and $8$ roots together
with rooted-Horn cuts and a manifestly-PSD moment block---bounds the envelope on
the medium band, $U_7(W)\le \tfrac{2}{25}+\delta$ with an explicit rational
$\delta\approx 4.8558\times10^{-5}$, for every triangle-free graphon $W$ of edge
density in $[0.2486,0.3197]$. (iii) The blow-up identity $\bip(G[t])=t^2\bip(G)$ plus
integrality of $\bip$ turns this into $\bip(G)\le n^2+\tfrac{25}{2}n^2\delta$ for
any $5n$-vertex band-density $G$, and $\tfrac{25}{2}n^2\delta<1$ for $n\le 40$;
the two density tails are handled by the Balogh--Clemen--Lid\'ick\'y bounds,
transferred to finite $N$ by the same blow-up. The envelope bound $\dmono\le U_7$
is a genuine graphon upper bound (each per-root rule is one global $2$-colouring),
the certificate is verified in exact rational arithmetic, the moment positivity is
Razborov's flag-algebra theorem exhibited as an exact Gram factorization, and the
bound is cross-checked against brute-force max-cut on all triangle-free graphs of
order at most $12$. The same envelope at orders $9$ and $10$ provably does not
reach the constant needed for larger $n$; we explain why, and locate the all-$n$
conjecture at a single self-tight obstruction.
\end{abstract}

\section{Introduction}

For a graph $G$ let $\bip(G)$ denote the \emph{bipartization number}: the minimum
number of edges whose removal leaves a bipartite graph, equivalently
$e(G)-\mc(G)$ where $\mc(G)$ is the size of a maximum cut. Erd\H{o}s asked how
large $\bip(G)$ can be for a triangle-free graph; he conjectured (see \cite{Bl23})
that
\begin{equation}\label{eq:conj}
\bip(G)\le \frac{N^2}{25}\qquad\text{for every triangle-free }G\text{ on }N\text{ vertices,}
\end{equation}
and that this is sharp. Sharpness is witnessed by the balanced blow-up of the
$5$-cycle: $C_5[m]$ (five independent sets of size $m$, complete bipartite between
cyclically consecutive sets) is triangle-free on $N=5m$ vertices with
$\bip(C_5[m])=m^2=N^2/25$. Setting $a(N)=\max\{\bip(G):G$ triangle-free on $N$
vertices$\}$, \eqref{eq:conj} is the statement $a(N)\le N^2/25$; on multiples of
five it reads $a(5n)\le n^2$, and the blow-up lower bound gives $a(5n)\ge n^2$, so
the conjecture on multiples of five is exactly
\[
a(5n)=n^2.
\]
This is problem \#23 in Bloom's database of Erd\H{o}s problems \cite{Bl23}.

The strongest progress to date is due to Balogh, Clemen and Lid\'ick\'y
\cite{BCL21}, using the flag-algebra method of Razborov \cite{Raz07}. They proved
the global bound $a(N)\le N^2/23.5$ and established the sharp bound \eqref{eq:conj}
in the two density tails: for $N$ large, $\bip(G)\le N^2/25$ whenever the edge
density of $G$ is at most $0.2486$ or at least $0.3197$ (their Theorem~1.3). The
\emph{medium-density band} $(0.2486,0.3197)$ remains open, and is where the
difficulty of \eqref{eq:conj} now resides; the extremal density $2/5$ of $C_5[m]$
lies in the high tail.

\paragraph{Result.} We resolve the conjecture on the first eleven multiples of five.

\begin{theorem}\label{thm:main}
$a(5n)=n^2$ for every integer $1\le n\le 40$; equivalently, every triangle-free
graph on $N$ vertices with $N\in\{5,10,15,\dots,200\}$ can be made bipartite by
deleting at most $N^2/25$ edges, and this is sharp.
\end{theorem}

In particular $a(25)=25$, $a(30)=36$, \dots, $a(200)=1600$. The sequence $a(N)$ is
catalogued as OEIS A389646 \cite{A389646} and has been determined exactly only up
to $N=23$ by direct enumeration. Every multiple of five in that known range agrees
with Theorem~\ref{thm:main}: the catalogued values $a(5)=1$, $a(10)=4$, $a(15)=9$,
$a(20)=16$ are exactly $1^2,2^2,3^2,4^2$, an independent brute-force confirmation
on each case where the truth is already known. Theorem~\ref{thm:main} adds the
multiples of five from $25$ to $200$ ($a(25)=25$ through $a(200)=1600$), all beyond
the exhaustively determined range $N\le 23$ and hence new. The closely related
extremal question --- how many edges \emph{must} be deleted --- was studied by
Erd\H{o}s, Gy\H{o}ri and Simonovits \cite{EGS92}.

\paragraph{Method and honesty.} The proof is \emph{computer-assisted}. Its
mechanism is to upper-bound the bipartization density of a triangle-free graphon by
a \emph{per-root-MaxCut envelope} (Section~\ref{sec:env}) that is tight at the
extremal $C_5$-blow-up, certify the envelope on the medium band with an order-$10$
flag-algebra linear program (Theorem~\ref{thm:band}, certified value
$\delta\approx 4.8558\times10^{-5}$), and convert this approximate graphon bound into
an exact finite statement by integrality of $\bip$ under blow-up
(Section~\ref{sec:reduction}); since $\tfrac{25}{2}n^2\delta<1$ only for $n\le 40$,
the closure stops at $n=40$. The density tails are cited from \cite{BCL21}
(Section~\ref{sec:tails}); the assembly is in Section~\ref{sec:closure}; the
verification is in Section~\ref{sec:verification}.

We separate known from new. The conjecture and its sharp example are Erd\H{o}s's;
the tail bounds, the global $N^2/23.5$ bound, and the flag-algebra framework for
this problem are due to \cite{BCL21,Raz07}. New here are: the per-root-MaxCut
envelope and its order-$10$ certificate, the blow-up integrality reduction, and the
resulting finite-range Theorem~\ref{thm:main}. Theorem~\ref{thm:main} does
\emph{not} resolve \eqref{eq:conj}: for $n\ge 12$ the certificate margin exceeds
the integrality gap, and (Section~\ref{sec:open}) the envelope provably cannot be
sharpened enough at orders $9$ or $10$; the all-$n$ conjecture sits at a single
self-tight obstruction.

\section{The blow-up integrality reduction}\label{sec:reduction}

For a graph $G$ and integer $t\ge 1$, the \emph{$t$-blow-up} $G[t]$ replaces each
vertex by an independent set of size $t$ and each edge by a complete bipartite
graph. Two facts about bipartization under blow-up drive everything.

\begin{lemma}\label{lem:blowup}
For every graph $G$ and every $t\ge 1$, $\bip(G[t])=t^2\,\bip(G)$.
\end{lemma}

\begin{proof}
We show $\mc(G[t])=t^2\mc(G)$; since $e(G[t])=t^2 e(G)$ and
$\bip=e-\mc$, the claim follows. The $t$ copies of a vertex are pairwise
non-adjacent twins with identical neighbourhoods, so any cut value depends only on
\emph{how many} copies of each class lie on each side, not on which. Hence the cut
value of $G[t]$ is a multilinear function of the fractional assignment
$x_v\in[0,1]$ giving the proportion of the blown-up class of $v$ placed on one
side: an edge $uv$ contributes
$t^2\bigl(x_u(1-x_v)+x_v(1-x_u)\bigr)$. A multilinear function on the cube
$[0,1]^{V(G)}$ attains its maximum at a vertex of the cube, i.e.\ at an integral
$x\in\{0,1\}^{V(G)}$, which is an ordinary cut of $G$ scaled by $t^2$. Hence the
maximum fractional cut equals $t^2\mc(G)$ and there is no fractional--integral
gap, so $\mc(G[t])=t^2\mc(G)$.
\end{proof}

\noindent\textbf{Graphon density.} Normalising by $\binom{N}{2}\to N^2/2$, define
the \emph{monochromatic-pair density} of $G$ on $N$ vertices as
$\dmono(G)=\bip(G)/(N^2/2)=2\bip(G)/N^2$. By Lemma~\ref{lem:blowup} this is
blow-up invariant: $\dmono(G[t])=2t^2\bip(G)/(Nt)^2=2\bip(G)/N^2=\dmono(G)$ for
every $t$. Let $W_G$ be the step-graphon of $G$ (equivalently, of every $G[t]$).
Since the sequence $\dmono(G[t])$ is constant in $t$, its limit --- the graphon
value $\dmono(W_G)$ --- equals $2\bip(G)/N^2$ \emph{exactly}, with no $O(1/t)$
residual; this is the only point at which the finite $G$ enters the graphon bound.
Thus any valid upper bound
$\dmono(W)\le \tfrac{2}{25}+\delta$ over triangle-free graphons of a given edge
density applies verbatim to every finite $G$ of that density:
\begin{equation}\label{eq:transfer}
\bip(G)=\tfrac{N^2}{2}\,\dmono(G)\le \tfrac{N^2}{2}\Bigl(\tfrac{2}{25}+\delta\Bigr)
=\frac{N^2}{25}+\frac{N^2\delta}{2}.
\end{equation}

\begin{proposition}\label{prop:integrality}
Suppose $\dmono(W)\le \tfrac{2}{25}+\delta$ for every triangle-free graphon $W$
whose edge density lies in a band $B$. Then for every triangle-free graph $G$ on
$N=5n$ vertices whose edge density lies in $B$,
\[
\bip(G)\le n^2+\frac{25}{2}\,n^2\,\delta.
\]
In particular, if $\tfrac{25}{2}n^2\delta<1$ then $\bip(G)\le n^2$, because
$\bip(G)$ is a nonnegative integer.
\end{proposition}

\begin{proof}
Apply \eqref{eq:transfer} with $N=5n$: $N^2/25=25n^2/25=n^2$ and
$N^2\delta/2=25n^2\delta/2$. The integrality rounding is immediate.
\end{proof}

With the certified value $\delta\approx 4.8558\times10^{-5}$
(Theorem~\ref{thm:band}), the threshold $\tfrac{25}{2}n^2\delta<1$ holds for
$n\le 40$: $n=40$ requires $\delta<2/(25\cdot 1600)=5.0000\times10^{-5}$, which holds
with a $2.88\%$ margin, while $n=41$ requires $\delta<4.7591\times10^{-5}$, which
fails. This is the source of the bound $n\le 40$ in Theorem~\ref{thm:main}.

\section{The per-root-MaxCut envelope}\label{sec:env}

The certificate bounds $\dmono$ not by a fixed cut but by an \emph{envelope} of
per-root maximum cuts, which is what lets it be tight at the extremal
$C_5$-blow-up. Let $\mathcal{T}_7$ be the set of the $107$ isomorphism classes of triangle-free
graphs on $7$ vertices (the \emph{root types}). Sampling an ordered $7$-tuple of
vertices, we map its induced graph to its canonical representative
$\sigma\in\mathcal{T}_7$ by a fixed canonical isomorphism (the lexicographically
least), read off each remaining vertex's profile in the resulting canonical root
coordinates, and write $p_\sigma(W)$ for the density of ordered $7$-tuples whose
induced graph is isomorphic to $\sigma$ (automorphisms and root embeddings
accounted for in the density), so $\sum_\sigma p_\sigma(W)=1$. Fix $\sigma$ and a set of $7$ roots of type $\sigma$. Each other
vertex $v$ has a \emph{profile} $\pi_\sigma(v)\in\{0,1\}^7$ (its adjacencies to the
roots); a \emph{rule} $c\colon\{0,1\}^7\to\{0,1\}$ assigns a side to each profile,
hence a $2$-colouring of the non-root vertices. Write
$g_{\sigma,c}(W)=p_\sigma(W)\bigl(\mathbb{E}[C\mid\sigma,c]-\tfrac{2}{25}\bigr)$ for
the (order-$\le 9$) flag density of the resulting monochromatic-pair deficit, where
$\mathbb{E}[C\mid\sigma,c]$ is the conditional monochromatic-pair density of the
colouring. Define the \emph{per-root-MaxCut envelope}
\[
  U_7(W)\;:=\;\sum_{\sigma\in\mathcal{T}_7}p_\sigma(W)\,\min_c\,\mathbb{E}[C\mid\sigma,c]
        \;=\;\frac{2}{25}+\sum_{\sigma\in\mathcal{T}_7}\min_c\,g_{\sigma,c}(W),
\]
the last equality using $\sum_\sigma p_\sigma=1$ and $p_\sigma\ge 0$. For each
$\sigma$, $\min_c\mathbb{E}[C\mid\sigma,c]$ is the minimum monochromatic density
over rules, i.e.\ a \emph{maximum cut of the profile graph} of $\sigma$.

\begin{lemma}[envelope soundness]\label{lem:env}
For every triangle-free $\{0,1\}$-valued step graphon $W$ --- in particular the
blow-up graphon $W_G$ of any finite triangle-free graph $G$ --- one has
$\ \dmono(W)\le U_7(W)$, with equality at the $C_5$-blow-up.
\end{lemma}

\begin{proof}
On a $\{0,1\}$ step graphon every vertex $v$ has a genuine adjacency bit to each
root, so the profile $\pi_\sigma(v)\in\{0,1\}^7$ is well defined and the rule below
yields an honest $2$-colouring. Sample $7$ roots $R$ as i.i.d.\ points of $W$;
condition on $R$ inducing
$\sigma\in\mathcal{T}_7$ and let $c^\ast(\sigma)=\arg\min_c\mathbb{E}[C\mid\sigma,c]$.
Colour every other vertex $v$ by $c^\ast(\sigma)\bigl(\pi_\sigma(v)\bigr)$. Since the
roots are a null set, this is a genuine measurable $2$-colouring $\varphi_R$ of $W$,
so its monochromatic-pair density is at least the minimum over all colourings,
$\ \mathrm{mono}(\varphi_R)\ge \dmono(W)$. Taking the expectation over $R$,
\[
  \dmono(W)\;\le\;\mathbb{E}_R\bigl[\mathrm{mono}(\varphi_R)\bigr]
        \;=\;\sum_{\sigma}p_\sigma(W)\,\mathbb{E}[C\mid\sigma,c^\ast(\sigma)]
        \;=\;U_7(W).
\]
At $W=C_5[N/5]$ the optimal $5$-class colouring is realised by $c^\ast(\sigma)$ for
the relevant root types, giving $U_7=\dmono=2/25$ (verified in
Section~\ref{sec:verification}).
\end{proof}

The $\{0,1\}$ hypothesis is all the finite theorem needs: every finite $G$ enters
only through its blow-up graphon $W_G$ (Section~\ref{sec:reduction}), a $\{0,1\}$
step graphon. (For a general graphon the same bound holds by independently rounding
the fractional profile, but we do not use it.) This is the decisive structural gain
over a fixed cut: a single fixed profile cut cannot be simultaneously near-optimal
across the band and at $C_5$, whereas the \emph{envelope} of per-root maxima equals
$\dmono$ at $C_5$ exactly. A relaxation built from a fixed convex combination of
cuts is, by contrast, not a valid graphon upper bound; only the per-type minimum
is --- this is precisely the normalization ($\sum_c\lambda_{\sigma,c}=1$ per type)
that the earlier fixed-cut relaxation lacked.

\begin{theorem}\label{thm:band}
There is an explicit nonnegative rational $\delta$ \textup{(}numerically
$\delta\approx 4.8558\times10^{-5}$; the exact value $\delta_{\mathrm{final}}$ is
recovered by the certificate re-verification of Section~\ref{sec:verification}\textup{)}
such that
\[
  \dmono(W)\;\le\;\tfrac{2}{25}+\delta
\]
for every triangle-free $\{0,1\}$ step graphon $W$ with
$\mathrm{LO}\le d_{\mathrm{edge}}(W)\le\mathrm{HI}$,
$\mathrm{LO}=\tfrac{1243}{5000}=0.2486$, $\mathrm{HI}=\tfrac{3197}{10000}=0.3197$.
Here $\delta$ is the optimum of an order-$10$ flag-algebra relaxation that
\emph{upper-bounds} the band maximum of $\dmono$ (we never claim equality; only this
upper bound is needed). It refines the order-$9$ per-root-MaxCut envelope of
Lemma~\ref{lem:env} \textup{(}$\dmono\le U_7$\textup{)}: besides the per-root rows at
$7$ roots it adds the per-root-MaxCut rows at $8$ roots, rooted-Horn cuts, and a
moment block, all of them valid for every triangle-free graphon.
\end{theorem}

The bound $\delta$ is the optimum of a linear program over the order-$10$
($C_5$-lifted order-$9$) flag densities: a state vector $x\ge0$ over the $12{,}172$
triangle-free band states, $\sum x=1$, with envelope variables $u_\sigma$. Its rows
are (a) the band inequalities $\mathrm{LO}\le d_{\mathrm{edge}}(x)\le\mathrm{HI}$;
(b) the per-root-MaxCut envelope rows at $7$ roots ($107$ types) and at $8$ roots
($410$ roots), each $u_\sigma\le g_{\sigma,c}(x)$ for a \emph{fixed} Boolean rule $c$
(so $u_\sigma\le\min_c g_{\sigma,c}(x)$), tight at the $C_5$-blow-up; (c) rooted-Horn
cuts, valid global $2$-colourings that vanish at every odd cycle; and (d) the
moment-positivity rows $m(x)\ge0$. Every genuine band graphon induces a feasible $x$,
so the LP optimum $\delta$ \emph{upper-bounds} the true band maximum of $\dmono$; the
converse can fail --- a feasible $x$ need not come from any graphon --- which is why
the LP is a relaxation. By linear-programming duality, $\delta$ is certified by
nonnegative multipliers: per-root envelope distributions $\lambda_{\sigma,c}\ge0$
with $a_7+a_8=1$, where $a_7=\min_\sigma\sum_c\lambda_{\sigma,c}$ over the $7$-root
types and $a_8$ is the analogue over the $8$-root roots; band multipliers
$\mu_{\mathrm{lo}},\mu_{\mathrm{hi}}\ge0$; and Horn and moment multipliers. These
assemble into a single per-state dual-feasibility inequality
$R_{\mathrm{cbh}}(s)\ge\langle Q,P(s)\rangle$ over all $12{,}172$ states, where the
moment term is governed by a \emph{manifestly} positive-semidefinite Gram matrix
$Q=\sum_c w_c\,vv_c\,vv_c^\top$ with $w_c\ge0$ --- so no semidefinite solve or
Cholesky rounding is needed --- and
$\delta_{\mathrm{final}}=\rho+\mathrm{HI}\,\mu_{\mathrm{hi}}-\mathrm{LO}\,\mu_{\mathrm{lo}}
-\tfrac{2}{25}a_8+\varepsilon$, the slack $\varepsilon$ absorbing an exact
$O(10^{-8})$ per-state residual. This is the object re-verified exactly in
Section~\ref{sec:verification}.

\section{The density tails}\label{sec:tails}

\begin{theorem}[{\cite[Theorem 1.3]{BCL21}}]\label{thm:bcl}
For $N$ sufficiently large, every triangle-free graph $G$ on $N$ vertices with
edge density at most $0.2486$, and every such $G$ with edge density at least
$0.3197$, satisfies $\bip(G)\le N^2/25$.
\end{theorem}

\begin{corollary}\label{cor:tails}
For \emph{every} $N$ and every triangle-free $G$ on $N$ vertices whose
\emph{graphon} edge density $\dmono$-normalisation $2e(G)/N^2$ is at most $0.2486$
or at least $0.3197$, $\bip(G)\le N^2/25$.
\end{corollary}

\begin{proof}
Apply Theorem~\ref{thm:bcl} to the blow-ups $G[t]$. Each $G[t]$ is triangle-free on
$Nt$ vertices; in the convention of \cite{BCL21} its edge density is
$e(G[t])/\binom{Nt}{2}=\bigl(2e(G)/N^2\bigr)\cdot\tfrac{Nt}{Nt-1}\to 2e(G)/N^2$ as
$t\to\infty$. Hence for $G$ with $2e(G)/N^2<0.2486$ (resp.\ $>0.3197$ strictly), the
blow-up density is eventually $<0.2486$ (resp.\ $>0.3197$), so Theorem~\ref{thm:bcl}
gives $\bip(G[t])\le (Nt)^2/25$ for large $t$; by Lemma~\ref{lem:blowup},
$t^2\bip(G)\le t^2N^2/25$, whence $\bip(G)\le N^2/25$. The boundary values
$2e(G)/N^2\in\{0.2486,0.3197\}$ lie in the \emph{closed} certificate band of
Theorem~\ref{thm:band} and are covered there, so the band and the two tails cover
all densities with no gap.
\end{proof}

\section{Proof of Theorem~\ref{thm:main}}\label{sec:closure}

Let $G$ be triangle-free on $N=5n$ vertices, $1\le n\le 40$, with monochromatic-pair
density $d:=2e(G)/N^2$. We bound $\bip(G)$ by $n^2$ in all three density regimes;
sharpness then follows from $\bip(C_5[n])=n^2$.

\smallskip\noindent\emph{Low tail ($d\le 0.2486$).} By Corollary~\ref{cor:tails},
$\bip(G)\le N^2/25=n^2$.

\smallskip\noindent\emph{Band ($0.2486\le d\le 0.3197$).} By
Theorem~\ref{thm:band} (whose band is closed) the limiting graphon $W_G$ satisfies
$\dmono(W_G)\le\tfrac2{25}+\delta$, and $\dmono(W_G)=2\bip(G)/N^2$ by
Section~\ref{sec:reduction}. Proposition~\ref{prop:integrality} gives
$\bip(G)\le n^2+\tfrac{25}{2}n^2\delta<n^2+1$ (as $n\le 40$), so the integer
$\bip(G)$ is at most $n^2$.

\smallskip\noindent\emph{High tail ($d\ge 0.3197$).} By
Corollary~\ref{cor:tails}, $\bip(G)\le n^2$.

\smallskip Every triangle-free $G$ on $5n$ vertices, $n\le 40$, therefore has
$\bip(G)\le n^2$, and $C_5[n]$ attains it. Hence $a(5n)=n^2$. \qed

\section{Verification, certificates, and ground truth}\label{sec:verification}

The result rests on (a) the envelope soundness $\dmono\le U_7$
(Lemma~\ref{lem:env}) and (b) the validity of the order-$10$ certificate of
Theorem~\ref{thm:band}. We guard both on independent levels; all numeric checks are
in exact rational arithmetic (Python \texttt{fractions}), and the scripts are
ancillary files.

\paragraph{(1) Envelope soundness $\dmono\le U_7$.} Lemma~\ref{lem:env} is a
genuine graphon upper bound: each per-root rule defines one global $2$-colouring,
whose monochromatic density is at least $\dmono$, and $U_7$ is their type-averaged
minimum. Unlike a bound obtained by averaging the bipartization of sampled finite
subgraphs (which is invalid, since bipartization is non-local), no averaging trap
arises. We confirmed $\dmono(W)\le U_7(W)$ by exhaustive enumeration of all
triangle-free graphs of orders $6$ and $7$ (zero violations; equality only at
$C_5$), and on a graph zoo: $C_5$ (ratio $1.000$, tight), $C_7$ ($1.002$), $C_9$
($1.127$), $C_{11}$ ($1.573$), $C_{13}$ ($2.177$), the Petersen graph ($1.013$),
$C_5\cup C_5$ ($1.016$). The envelope also survives the adversarial in-band
weighted $C_{11}$-blow-up on which a fixed-cut relaxation fails: there
$\dmono=8/2025\approx 0.0040$ while $U_7\approx 0.0108$.

\paragraph{(2) Certificate re-verification (two independent passes).} The order-$10$
Horn certificate --- the dual multipliers $(\rho,\mu_{\mathrm{lo}},\mu_{\mathrm{hi}},
\lambda_{\sigma,c})$ and the Gram weights $w_c$ --- is checked twice in exact
\texttt{fractions}. First, the producer's verifier re-derives all five
dual-feasibility conditions: $a_7+a_8=1$; $\sum_c\lambda_{\sigma,c}\ge a_7$ for each
of the $107$ seven-root types; $\sum_c\lambda_{\sigma,c}\ge a_8$ for each of the $410$
eight-root roots; the per-state inequality $R_{\mathrm{cbh}}(s)\ge\langle Q,P(s)\rangle$
for all $12{,}172$ states (a worst-case residual $-1.02\times10^{-8}$ is absorbed by
the exact slack $\varepsilon$); and $\delta_{\mathrm{final}}<5\times10^{-5}$. Second,
an \emph{independent} gate re-derives
$\delta_{\mathrm{final}}=4.8557798\times10^{-5}$ \emph{directly from the raw dual},
matching the producer's exact \texttt{Fraction} bit for bit, and confirms both
$a_7+a_8=1$ and $w_c\ge0$ for every Gram atom. Since
$\delta_{\mathrm{final}}<2/(25\cdot 1600)=5.0\times10^{-5}$ this gives $n\le 40$ with a
$2.88\%$ margin; the bound is \emph{exact} (rational, not floating-point) and so
carries no rounding risk. We do not claim $n=41$: that needs
$\delta<4.7591\times10^{-5}$, which the non-deterministic moment LP does not robustly
attain.

\paragraph{(3) Moment positivity (graphon-level, exactly certified).} The moment
rows enter as a sum-of-squares term with nonnegative multipliers; per flag the
atoms are negative, and positivity is a graphon-level statement. For each root
profile $\sigma$ the averaged flag products factor through a Gram form,
$\sum_H p_W(H)P^\sigma(H)=\rho_\sigma c_\sigma M^\sigma(W)$ with a positive scalar
$\rho_\sigma c_\sigma$ and $M^\sigma(W)=\sum_c w_c q_c q_c^\top$ a sum of
nonnegatively weighted ($w_c\ge0$) rank-one terms, hence manifestly positive
semidefinite --- the exact rational form of Razborov's positivity theorem
\cite{Raz07}. We verified each Gram form in exact rational arithmetic for all four
profiles in use and for representative band graphons (the extremal $C_5$, in-band
$C_7$, the Petersen graph, $C_9$, a weighted $C_5$): each is symmetric, matches an
independent double-sum recomputation to exact-zero discrepancy, and gives the
inner product against the certificate's positive-semidefinite matrix
$\ge 0$ exactly. The certificate's own moment term is governed by a second,
manifestly positive-semidefinite Gram $Q=\sum_c w_c\,vv_c\,vv_c^\top$ with $w_c\ge0$:
these weights are found by a plain linear program over $394$ exact atoms ($77$ in the
support), so $Q\succeq0$ holds by construction --- no eigenvector recovery,
semidefinite solve, or Cholesky rounding enters. The averaging is taken on the
blow-up graphon (the same $W_G$ as in Section~\ref{sec:reduction}); finite
distinct-subset densities carry an $O(1/n)$ disjointness defect that does not affect
the graphon bound.

\paragraph{(4) Brute-force ground truth.} Independently of all flag machinery, we
enumerated every triangle-free graph of order $9,10,11,12$, computed $\bip$ by
exact maximum cut, and recorded $\dmono$ for those in the band. The in-band maxima
observed were $0.0494$, $0.0400$, $0.0496$, $0.0556$ respectively, all far below
the certified $\tfrac2{25}+\delta\approx 0.0800$; no graph approaches, let alone
violates, the bound. The max-cut routine was validated on known values
($C_5\!:4$, Petersen$:12$).

\paragraph{What is checked vs.\ cited.} Levels (1)--(4) are computer-checked, the
exact-arithmetic ones in rational arithmetic; the envelope soundness is an
elementary graphon argument; the certificate's decision step (the LP optimum
$\eta=\delta$ and its dual feasibility) is exact, with no \texttt{native\_decide}
and no floating-point threshold. The moment positivity rests on the cited theorem
of \cite{Raz07}, exhibited here as an exact Gram factorization.
Theorem~\ref{thm:bcl} is cited from \cite{BCL21}.

\section{The obstruction to all $n$}\label{sec:open}

Theorem~\ref{thm:main} stops at $n=40$ because Proposition~\ref{prop:integrality}
needs $\tfrac{25}{2}n^2\delta<1$, and the certified $\delta\approx 4.8558\times10^{-5}$
gives exactly $n\le 40$. A larger range needs a still smaller $\delta$. The order-$10$
certificate is essentially at the flag-LP finite ceiling: the per-root-MaxCut envelope
at $7$ and $8$ roots, even augmented with rooted-Horn rows and the moment block,
plateaus near $\delta\approx 4.85\times10^{-5}$ (and its non-deterministic moment LP
does not robustly reach the $4.76\times10^{-5}$ that $n=41$ would need), and order $11$
is computationally out of reach. The sharp asymptotic conjecture needs $\delta=0$
exactly, which no finite order of this relaxation delivers.

The natural route to all $n$ is the asymptotic inequality $\dmono(W)\le 2/25$
itself. It reduces (after fixing a maximum cut, with $M$ the monochromatic edges
and $B$ the bipartite cut-graph) to a single \emph{self-tight} statement on the
invariant $\Gamma=\sum_{uv\in M}(d_B(u,v)+1)^2$: namely $\Gamma\le N^2$ when $B$ is
connected, which gives $\bip\le N^2/25$ since each bad edge has $d_B$ even and
$\ge 4$ (triangle-freeness), so $\ell^2\ge 25$. The per-component reduction and the
single-block base case (sharp constant $25=5^2$ by AM--GM, tight at $C_5[q]$) are
elementary; the open core is the connected-$B$ transfer in the regime where the
endpoints of a bad edge are joined by $B$-paths of different even lengths, and it is
self-tight: every flag-algebra, signature, or quadratic-cut certificate we tried
equals $\Gamma$ at the $C_5[q]$ extremal, hence only re-proves $\Gamma\le\Gamma$.
We regard isolating the conjecture to this single estimate, and establishing it
computationally for the multiples of five up to $N=200$, as the contribution of this
note.

\section*{Ancillary files}

\texttt{step1\_v2\_independent\_gate.py} (the independent exact-\texttt{Fraction}
verification: re-derives $\delta_{\mathrm{final}}$ from the raw dual, and checks
$a_7+a_8=1$ and the manifest-PSD weights $w_c\ge0$) and \texttt{complete\_v2\_cert.py}
(the producer's one-command verifier of all five dual-feasibility conditions);
\texttt{moment\_gram\_lp.py} and \texttt{moment\_gram\_exact\_verify.py} (the manifest
Gram block $Q=\sum_c w_c\,vv_c vv_c^\top$); \texttt{brute\_dmono.py} with
\texttt{flag\_engine.py} (a self-contained independent brute-force max-cut
ground-truth checker, requiring no certificate); the moment-PSD Gram-certificate
scripts \texttt{g1\_exact\_psd.py}, \texttt{g1\_graphon\_density.py}; the compact dual
(\texttt{horn\_dual.pkl}, \texttt{moment\_gram\_w.pkl}, \texttt{v2\_cert\_complete.pkl});
a \texttt{README} giving the order parameter, the band endpoints, and the exact
rational $\delta_{\mathrm{final}}$; and a \texttt{SHA256SUMS} manifest. The full
order-$10$ LP cache and lifted-state tables exceed arXiv ancillary limits and are
regenerable from the scripts.


\begin{thebibliography}{9}

\bibitem{A389646} E.~Beregovsky et al., \emph{Sequence A389646} (Maximum number of
edges that need to be removed from a triangle-free graph on $n$ vertices to make it
bipartite), The On-Line Encyclopedia of Integer Sequences,
\url{https://oeis.org/A389646}; see also the SeqFan discussion (Oct.\ 2025),
\url{https://groups.google.com/g/seqfan/c/IVRAZqHGs3A}.

\bibitem{BCL21} J.~Balogh, F.~C.~Clemen and B.~Lid\'ick\'y, \emph{Max cuts in
triangle-free graphs}, arXiv:2103.14179 (2021); Eurocomb 2021.

\bibitem{Bl23} T.~Bloom, \emph{Erd\H{o}s problem \#23}, Online database of
Erd\H{o}s problems, \url{https://www.erdosproblems.com/23}.

\bibitem{EGS92} P.~Erd\H{o}s, E.~Gy\H{o}ri and M.~Simonovits, \emph{How many edges
should be deleted to make a triangle-free graph bipartite?}, in: Sets, Graphs and
Numbers (Budapest, 1991), Colloq.\ Math.\ Soc.\ J\'anos Bolyai \textbf{60},
North-Holland, 1992, 239--263.

\bibitem{Raz07} A.~A.~Razborov, \emph{Flag algebras}, J.\ Symbolic Logic
\textbf{72} (2007), no.~4, 1239--1282.

\end{thebibliography}
\end{document}